\documentclass[11pt]{amsart}

\pdfoutput=1
\usepackage{amsmath, amsthm, amssymb,slashed,stmaryrd}
\usepackage{ifpdf}
\usepackage[pdftex]{graphicx}
\usepackage{tikz}
\usetikzlibrary{matrix,arrows,calc}
\usepackage{mathtools}
\usepackage[pdftex,plainpages=false,hypertexnames=false,pdfpagelabels]{hyperref}
\usepackage{tikz-cd}
\usetikzlibrary{arrows,decorations.markings}
\usetikzlibrary{shapes.geometric}

\usepackage{bm}
\usepackage{accents}
\usepackage[super]{nth}
\usepackage{comment}

 \theoremstyle{plain}

\numberwithin{equation}{subsection}
 \newtheorem{thm}[subsection]{Theorem}

\newtheorem{lem}[subsection]{Lemma}
\newtheorem{prop}[subsection]{Proposition}
\newtheorem{cor}[subsection]{Corollary}
\theoremstyle{definition}

\newtheorem{rmk}[subsection]{Remark}
\newtheorem{question}[subsection]{Question}

\newtheorem{claim}[subsection]{Claim}
\newtheorem*{conjecture*}{Conjecture}
\newtheorem*{theorem*}{Theorem}
\newtheorem*{claim*}{Claim}
\newtheorem*{corollary*}{Corollary}

\def\beq{\begin{eqnarray}}
\def\eeq{\end{eqnarray}}
 \newcommand{\bp}{\begin{proof}[Proof]}
 \newcommand{\ep}{\end{proof}}

\DeclareSymbolFont{bbold}{U}{bbold}{m}{n}
\DeclareSymbolFontAlphabet{\mathbbold}{bbold}

\def\Hom{{\sf Hom}}

\def\Spec{{\rm Spec}}
\def\Spf{{\rm Spf}}

\def\Aut{{\rm Aut}}

\def\K3{{\rm K3}}

\def\Fil{{\rm Fil}}

\def\GL{{\rm GL}}

\newcommand{\defeq}{\vcentcolon=}

\makeatletter
\newcommand{\colim@}[2]{%
  \vtop{\m@th\ialign{##\cr
    \hfil$#1\operator@font colim$\hfil\cr
    \noalign{\nointerlineskip\kern1.5\ex@}#2\cr
    \noalign{\nointerlineskip\kern-\ex@}\cr}}%
}
\newcommand{\colim}{%
  \mathop{\mathpalette\colim@{}}\nmlimits@
}
\makeatother

\newcommand\nc{\newcommand}
\nc\mf\mathfrak
\nc\mc\mathcal
\nc\mb\mathbb
\nc\msf\mathsf

\usepackage[
backend=biber,
style=alphabetic,
sorting=nyt,
maxnames=50
]{biblatex}

\begin{document}

\title{Lifts of supersingular abelian varieties with small Mumford--Tate groups}

\author{Yeuk Hay Joshua Lam}
\author{Abhishek Oswal}

\date{\today}

\begin{abstract}  
We investigate to what extent an abelian variety over a finite field can be lifted to one in characteristic zero  with small Mumford-Tate group. We prove that supersingular abelian surfaces, respectively threefolds, can be lifted to ones isogenous to a square, respectively product, of elliptic curves. On the other hand, we show that supersingular abelian threefolds cannot be lifted to one isogenous to the cube of an elliptic curve over the Witt vectors. 
\end{abstract}

\maketitle 
\setcounter{tocdepth}{1}
\tableofcontents

\section{Introduction}
Given an abelian variety $A_0$ over  $\overline{\mb{F}}_p$, the question of what kind of lifts to characteristic zero $A_0$ admits has been studied for a long time. For example, the famous result of Honda shows that any such $A_0$ admits a lift to a complex multiplication (CM) abelian variety, \emph{up to isogeny}; more precisely, there exists an abelian variety $A_0'/\overline{\mb{F}}_p$ isogenous to $A_0$ which has a CM lift. A vast generalization of this result to a large class of  Shimura varieties was achieved by Kisin \cite{kisin}. On the other hand, the question of CM liftings on the nose, i.e. not replacing $A_0$ by something isogenous to it,  was considered by Oort \cite{oort92}, who gave examples of $A_0$ not admitting any CM lifts; throughout this paper, by a lift we will mean this stronger sense of lifting without allowing first an isogeny. In a recent work \cite{heckeorbit} of the first author with Kisin, Shankar, and Srinivasan, it was proven that, in fact, for a fixed $g$, only finitely many supersingular abelian varieties of dimension $g$ admit CM lifts. A similar result for the other Newton polygons, i.e. intermediate between ordinary and supersingular, as well as for K3 surfaces, was also obtained\footnote{the precise formulation is that the points admitting CM lifts lie on finitely many central leaves; we refer the reader to \cite{heckeorbit}[Theorem 1.3]}. We note that ordinary abelian varieties always admit CM lifts, namely the \emph{canonical lifts} over $W(\overline{\mb{F}}_p)$. For almost ordinary abelian varieties over $\overline{\mb{F}}_p$, in work of the second author with Shankar \cite{almost-ordinary}, there is a similar construction of CM lifts of such abelian varieties to `slightly ramified' extensions of the ring of Witt vectors $W(\overline{\mb{F}}_p)$: we refer the reader to \S 2 of loc.cit.\ for the precise statement.

Therefore it is, in a vague sense, supersingularity as well as the ramification of the lifting base that prevents an abelian variety from having CM lifts. We then have the following natural

\begin{question}\label{question:intro}
What is the smallest possible Mumford-Tate group of a lift of $A_0$? What about for lifts over $W(\overline{\mb{F}}_p)$?
\end{question}
As stated, Question~\ref{question:intro} is a little imprecise; as a first approximation, one can ask whether such an $A_0$ can be lifted so that some of its endomorphisms also lift, with the CM case being the best one can do and  not always possible. In this work we begin to investigate this question, and indicate that the answer may be rather subtle.
\begin{thm}\label{thm:main}\hfill
\begin{enumerate} 
\item 
Any supersingular abelian surface $A_0$ over $\overline{\mb{F}}_p$ admits a lift which is isogenous to the square of an elliptic curve.
\item 
Any supersingular abelian threefold $A_0$ over $\overline{\mb{F}}_p$  admits a lift which is isogenous to a product of three elliptic curves.
\item 
Not every supersingular abelian threefold admits a lift over $W(\overline{\mb{F}}_p)$ isogenous to the cube of an elliptic curve.
\end{enumerate}
\end{thm}
In the language of Question~\ref{question:intro}, this means that abelian surfaces can always be lifted to ones with Mumford-Tate group contained in $\GL_2$, and for threefolds $(\GL_2)^3$. It would also be interesting to investigate other kinds of reduction in Mumford-Tate groups: for example, for the case of surfaces the other possibilities are that the lift $\mc{A}$ is isogenous to a product $E_1\times E_2$ where $E_1$ has CM, or that $\mc{A}$ has quaternionic multiplication.
\begin{rmk}
Our results concern only liftings to the Witt vectors $W(\overline{\mb{F}}_p)$. It would of course be interesting to understand whether the answer changes, in the cases we study, if we allow liftings to arbitrary discrete valuation rings.
\end{rmk}
\begin{rmk}
In light of Theorem~\ref{thm:main}, it seems reasonable to conjecture that, at least over $W(\overline{\mb{F}}_p)$, every supersingular abelian variety admits a lift isogenous to a product of elliptic curves, and that most supersingular abelian varieties do not admit lifts isogenous to the self power of an elliptic curve.
\end{rmk}
We now give an  idea of the proofs. By Grothendieck-Messing and Serre-Tate theory, to exhibit a lift of an abelian variety it suffices to give a filtration on the Dieudonn\'e module lifting the kernel of Frobenius modulo $p$. Therefore, for parts (1) and (2) of Theorem~\ref{thm:main}, it suffices to find such filtrations  such that the resulting filtered Dieudonn\' e module is isogenous to that of a product of elliptic curves. We achieve this   by a direct computation. For the third part of Theorem~\ref{thm:main}, we show that if an abelian threefold does admit  a lift isogenous to a product of elliptic curves, then there exist non-trivial algebraic relations between the parameters of its Dieudonn\'e module, which allows us to conclude. Crucial to our computations are explicit parametrizations of the supersingular loci for abelian surfaces and threefolds, which may be of independent interest.

\subsection{Acknowledgements} It is a pleasure to thank Ananth Shankar for asking us this question and to Ananth Shankar and Padma Srinivasan for several invaluable conversations. A.O. thanks Caltech, Pasadena and IAS, Princeton for excellent working conditions.

\subsection{Notations} Throughout, $p$ denotes a fixed odd prime number, $W \defeq W(\overline{\mb{F}}_p)$ denotes the ring of Witt vectors with coefficients in $\overline{\mb{F}}_p,$ $\sigma : W \rightarrow W$ the ring automorphism of $W$ that lifts the absolute Frobenius of $\overline{\mb{F}}_p.$ We denote by $v : W[1/p] \twoheadrightarrow \mb{Z}\cup \{\infty\},$ the (additive) discrete valuation on $W[1/p]$. 

\section{Reduction to a computation with Dieudonn\'e modules}
Let $A_0$ be a supersingular abelian variety over $\overline{\mb{F}}_p$ of dimension $g$. Let $n_1,\ldots,n_r$ be positive integers such that $\sum_{1\leq j \leq r} n_j = g$. The question of whether there exists an abelian scheme $\mathcal{A}$ over $\Spec(W)$ lifting $A_0$ and elliptic curves $\mathcal{E}_1,\ldots,\mathcal{E}_r$ over $\Spec(W)$ such that $\mathcal{A}$ is isogenous over $W$ to $\mathcal{E}_1^{n_1} \times \ldots \times \mathcal{E}_r^{n_r}$, reduces (via Grothendieck--Messing theory and the theorem of Serre--Tate) to an explicit computation with Dieudonn\'e modules. 

More precisely, let $(\mathbb{D},F,V)$ be the Dieudonn\'e module of $A_0[p^\infty]$. Let $E_0$ be a supersingular elliptic curve over $\overline{\mb{F}}_p$ and let $(\mathbb{D}_0,F_0,V_0)$ be the Dieudonn\'e module of $E_0[p^\infty]$. 

\begin{prop}\label{reduction-to-d-modules}
With notations as above, there is an abelian scheme $\mathcal{A}$ over $\Spec(W)$, lifting $A_0$ and elliptic curves $\mathcal{E}_1,\ldots,\mathcal{E}_r$ over $\Spec(W)$ such that $\mathcal{A}$ is isogenous to $\mathcal{E}_1^{n_1}\times \ldots \mathcal{E}_r^{n_r}$ over $W$ if and only if there exist filtrations $\Fil \subset \mathbb{D}$, and $\Fil_1,\ldots,\Fil_r \subset \mathbb{D}_0$ such that the following hold:
\begin{itemize}
    \item The quotient modules $\mathbb{D}/\Fil$ and $\mathbb{D}_0/\Fil_j$ for $1 \leq j \leq r$ are torsion-free over $W$.
    \item The image of the filtration $\Fil$ (respectively $\Fil_j$ for $1 \leq j \leq r$) in $\mathbb{D}/p\mathbb{D}$ (respectively $\mathbb{D}_0/p\mathbb{D}_0$) is the kernel of the Frobenius $F$ (respectively $F_0$) modulo $p$.
    \item There is an isogeny of filtered Dieudonn\'e modules $\varphi : (\mathbb{D},\Fil,F,V) \rightarrow (\mathbb{D}_0,\Fil_1,F_0,V_0)^{\oplus{n_1}}\bigoplus \ldots \bigoplus (\mathbb{D}_0,\Fil_r,F_0,V_0)^{\oplus{n_r}}$. 
\end{itemize}
\end{prop}
\begin{proof}
    It is clear that if there is an abelian scheme $\mathcal{A}$ and elliptic curves  $\mathcal{E}_1,\ldots,\mathcal{E}_r$ as in the statement of the proposition, then there do exist filtrations and an isogeny $\varphi$ as above. 
    
    Conversely, suppose that there exist filtrations and an isogeny $\varphi$ as above. The isogeny $\varphi$ corresponds to an isogeny, again denoted by $\varphi$, of $p$-divisible groups $A_0[p^\infty]\xrightarrow{\varphi} E_0^g[p^\infty]$ over $\overline{\mb{F}}_p$.
    We first prove the following lemma.
    \begin{lem}
    There exists an automorphism $\alpha$ of the $p$-divisible group $A_0[p^\infty]$ and an isogeny $\gamma : A_0 \rightarrow E_0^g$ over $\overline{\mb{F}}_p$ such that $\varphi \circ \alpha = \gamma[p^\infty]$.   
    \end{lem}
    \begin{proof}[Proof of Lemma]
        Consider the dual isogeny $\varphi^\vee : E_0^g[p^\infty]\rightarrow A_0[p^\infty]^\vee$. First we note that $\varphi^\vee$ is defined over a finite field. Indeed, since $A_0$ is supersingular, we may pick some isogeny $\beta: A_0^{\vee}\rightarrow E_0^g$. Then $\beta \circ \varphi^{\vee}$ is a self isogeny of $E_0^g[p^{\infty}]$, i.e. an element of $M_g(D)$ \footnote{here $M_g$ denotes the algebra of $g\times g$ matrices} with $D$ being the non-split quaternion algebra over $\mb{Q}_p$, and all of the elements of this algebra are defined over a finite field. Therefore $\varphi^{\vee}$ is also defined over a finite field, as claimed.
        
        Now by Tate's isogeny theorem over finite fields (see for instance \cite[Theorem A.1.1.1]{conrad-chai-oort}), there is a sequence of elements $\gamma_i^\vee \in \Hom_{\mathrm{Ab.Var}/\overline{\mb{F}}_p}(E_0^g,A_0^\vee),$ such that $\gamma_i^\vee[p^\infty]$ converges to $\varphi^\vee$ in the $p$-adic topology. For $i$ sufficiently large, $\ker(\gamma_i^\vee[p^\infty]) = \ker(\varphi^\vee),$ and thus there is an automorphism $\alpha^\vee \in \Aut_{p\text{-div.gr}/\overline{\mb{F}}_p}(A_0[p^\infty]^\vee)$ such that $\alpha^\vee \circ \varphi^\vee = \gamma_i^\vee[p^\infty].$
    \end{proof}
    
The automorphism $\alpha$ of the $p$-divisible group $A_0[p^\infty]$ induces an automorphism of the Dieudonn\'e module $(\mathbb{D},F,V)$, again denoted by $\alpha$. Let $\Fil_0 := \alpha^{-1}(\Fil).$ The image of $\Fil_0$ in $\mathbb{D}/p\mathbb{D}$ is the kernel of $F$ modulo $p$. By Grothendieck--Messing theory\cite[Ch.\,V, Theorem 1.6]{messing1972crystals} and the theorem of Serre--Tate \cite[Theorem 1.2.1]{katz1981serre}, the isogeny of filtered Dieudonn\'e modules, \[\varphi \circ \alpha : (\mathbb{D},\Fil_0,F,V) \rightarrow (\mathbb{D}_0,\Fil_1,F_0,V_0)^{\oplus{n_1}} \bigoplus \ldots \bigoplus (\mathbb{D}_0,\Fil_r,F_0,V_0)^{\oplus{n_r}},\] gives rise to formal abelian schemes $\mathfrak{A}$ and $\mathfrak{E}_1,\ldots, \mathfrak{E}_r$ over $\Spf(W)$ lifting $A_0$ and $E_0$ respectively, and an isogeny $\widetilde{\gamma} : \mathfrak{A}  \rightarrow \mathfrak{E}_1^{n_1} \times \ldots \times \mathfrak{E}_r^{n_r}$ lifting $\gamma : A_0  \rightarrow E_0^g.$ The formal abelian schemes $\mathfrak{E}_i$ are all algebraizable, that is there are elliptic curves $\mathcal{E}_i$ over $\Spec(W)$ whose $p$-adic completion is $\mathfrak{E}_i$. The isogeny $\widetilde{\gamma}$ induces a polarization on $\mathfrak{A}$, and thus $\mathfrak{A}$ and the isogeny $\widetilde{\gamma}$ arise as the $p$-adic completion of an isogeny $\mathcal{A} \rightarrow \mathcal{E}_1^{n_1} \times \ldots \times \mathcal{E}_r^{n_r} $ over $\Spec(W)$ for an abelian scheme $\mathcal{A}$ over $\Spec(W)$ lifing $A_0$. This completes the proof of the proposition. 
\end{proof}

\section{Explicit Frobenius and Verschiebung}
\subsection{Abelian surfaces}

Let $E_0$ be a supersingular elliptic curve over $\overline{\mb{F}}_p$ and let  $(\mathbb{D}_0,F_0,V_0)$ denote the Dieudonn\'e module of $E_0[p^\infty].$ Then there is a basis $\{e,f\}$ of $\mathbb{D}_0$ with respect to which $F_0 = \begin{pmatrix}
0 & p \\
1 & 0
\end{pmatrix} = V_0.$ Let $\Fil_0 \subset \mathbb{D}_0$ be an admissible filtration, i.e. a filtration lifting the kernel of Frobenius modulo $p$. We record the following lemma whose proof is straightforward and left to the reader.
\begin{lem}\label{lem:sublattices-of-D0}
For every integer $n \geq 0$, there is a unique $W$-submodule $L(n) \subset \mathbb{D}_0$ of rank $2$, stable under $F_0$ and $V_0$, and such that the torsion $W$-module $\mathbb{D}_0/L(n)$ has length $n$. If $n = 2r$ is even, then $L(n) = \langle p^re, p^r f\rangle$. If $n=2r+1$ is odd, then $L(n) = \langle p^{r+1}e, p^r f \rangle$. 
Furthermore, $\Fil \cap L(n)$ is an admissible filtration of $\left(L(n),F_0\vert_{L(n)},V_0\vert_{L(n)}\right)$ if and only if $n$ is even.
\end{lem}

In this subsection, we let $(\mathbb{D}',F',V') \defeq (\mathbb{D}_0,F_0,V_0)^{\oplus{2}}$, and let $e_1' \defeq (e,0), f_1' \defeq (f,0), e_2'\defeq (0,e),f_2'\defeq (0,f)$. Let us also denote by $\mathrm{pr}_1$ (respectively $\mathrm{pr}_2$) the projection $\mathbb{D}'\rightarrow \mathbb{D}_0$ to the first (respectively second) coordinate. 

\begin{prop}\label{prop:explicit-isogeny-2d}
Suppose $\varphi : (\mathbb{D},F,V) \rightarrow (\mathbb{D}',F',V')$ is an isogeny of Dieudonn\'e modules. Then there is a $W$-basis $\{e_1,f_1,e_2,f_2\}$ of $\mathbb{D}$ with respect to which the actions of the Frobenius and Verschiebung are given by
    \begin{align}\label{eq:F,V,2d}
       F=\begin{pmatrix}
0 & p & 0 & 0 \\
1 & 0 & 0 & x \\
  &   &  0 & p\\
  &   &   1 & 0
\end{pmatrix}, \ \ V =\begin{pmatrix}
0 & p & -\sigma^{-1}(x) & 0 \\
1 & 0 & 0 & 0 \\
  &   &  0 & p\\
  &   &   1 & 0
\end{pmatrix}. 
    \end{align}  for some $x,y \in W$ and such that $\varphi(\langle e_1,f_1\rangle) \subseteq \langle e_1',f_1'\rangle$ and  $\varphi(\langle e_1,f_1, e_2,f_2\rangle) \subseteq \langle e_1',f_1',e_2',f_2'\rangle$. 

Furthermore, if there exist admissible filtrations $\Fil \subset \mathbb{D}$ and $\Fil_1, \Fil_2 \subset \mathbb{D}_0$ such that $\varphi(\Fil) \subseteq \Fil_1\oplus \Fil_2$, then we may additionally insist that \begin{itemize}
    \item $\varphi(e_1) = p^r e_1'$ and $\varphi(f_1) = p^r f_1'$ for some $r \geq 0$, 
    \item $\varphi(e_2) \in p^s e_2' + \langle e_1', f_1' \rangle$ and $\varphi(f_2) \in p^s f_2' + \langle e_1', f_1' \rangle$, for some $s \geq 0$. 
\end{itemize} 
\end{prop}
\begin{proof}
By Lemma~\ref{lem:sublattices-of-D0}, $\varphi(\mathbb{D})\cap \langle e_1', f_1' \rangle$ is either equal to $\langle p^{r+1}e_1', p^rf'\rangle$ or $\langle p^r e_1', p^r f_1'\rangle$ for some $r \geq 0$. If there furthermore exist filtrations $\Fil$ and $\Fil_1, \Fil_2$ as above, then this would imply that $\varphi(\Fil)\cap\langle e_1', f_1' \rangle = \Fil_1 \cap \varphi(\mathbb{D})$ and thus that $\Fil_1 \cap \varphi(\mathbb{D})$ is an admissible filtration of $\varphi(\mathbb{D}) \cap \langle e_1', f_1'\rangle$. Therefore, by Lemma \ref{lem:sublattices-of-D0} again,  $\varphi(\mathbb{D})\cap \langle e_1', f_1' \rangle$ must equal $\langle p^r e_1', p^r f_1'\rangle$ for some $r \geq 0$. 

Thus, we may pick a basis $\{e_1,f_1\}$ of $\varphi^{-1}(\langle e_1', f_1'\rangle)$ such that $F(e_1) = f_1$ and $F(f_1) = pe_1$, and in the case that there exist filtrations $\Fil$ and $\Fil_1,\Fil_2$ as above, we may also insist that $\varphi(e_1) = p^r e_1'$ and $\varphi(f_1) = p^rf_1'$. 

By Lemma \ref{lem:sublattices-of-D0}, $\mathrm{pr}_2 \circ\varphi(\mathbb{D})$ is either equal to $\langle p^{s+1}e_2', p^s f_2'\rangle$ or $\langle p^s e_2', p^s f_2' \rangle$ for some $s \geq 0$, and furthemore if there exist filtrations as above, then only the second of the above two cases can occur. Therefore, we may complete $\{e_1,f_1\}$ to a basis $\{e_1,f_1,e_2'',f_2''\}$ of $\mathbb{D}$ such that the Frobenius $F$ in this basis is of the form \[F=\begin{pmatrix}
0 & p & a & b \\
1 & 0 & c & d \\
  &   &  0 & p\\
  &   &   1 & 0
\end{pmatrix}\] for some $a, b, c, d \in W$. We further note that the condition that $F \circ V = p$ implies that $b \in pW$. In the case that there are filtrations as in the statement of the proposition, we may additionally insist that $\varphi(e_2'') = p^s e_2'$ and $\varphi(f_2'') = p^s f_2'$ for some $s \geq 0$.

Pick a $\lambda \in W$, such that $\lambda - \sigma^2(\lambda) = \sigma(c)+ b/p$. 
Set $e_2 \defeq e_2'' + \lambda e_1 $ and $f_2 \defeq f_2'' + a e_1 + \left( \sigma(\lambda)+c\right)f_1.$ Then one may check that the basis $\{e_1,f_1,e_2,f_2\}$ of $\mathbb{D}$, satisfies the requirements of the proposition, with $x=d+\sigma(a)$.
\end{proof}

Since every supersingular abelian surface over $\overline{\mb{F}}_p$ is isogenous to $E_0\times E_0$, we obtain the following corollary.
\begin{cor}\label{cor:dieudonne-module-surface}
Let $A_0$ be a supersingular abelian surface over $\overline{\mb{F}}_p$. Then there exists a basis of the Dieudonn\'e module of $A_0[p^\infty]$ with respect to which the actions of Frobenius and Verschiebung are given as in \ref{eq:F,V,2d} of Proposition \ref{prop:explicit-isogeny-2d} above for some $x \in W$.
\end{cor}

\subsection{Abelian threefolds}\label{sec:lifitng-threefolds}

In this subsection we let \[(\mathbb{D}',F',V') \defeq (\mathbb{D}_0,F_0,V_0)^{\oplus{3}},\] and we set $e_1' \defeq (e,0,0), f_1' \defeq (f,0,0), e_2'\defeq (0,e,0), f_2'\defeq (0,f,0), e_3'\defeq (0,0,e)$ and $f_3'\defeq(0,0,f)$. 

We now record the following version of Proposition~\ref{prop:explicit-isogeny-2d} for supersingular abelian threefolds, the  proof of which being similar to that of Proposition \ref{prop:explicit-isogeny-2d}, is omitted.

\begin{prop}\label{prop:explicit-isogeny}
Suppose $\varphi : (\mathbb{D},F,V) \rightarrow (\mathbb{D}',F',V')$ is an isogeny of Dieudonn\'e modules. Then there is a $W$-basis $\{e_1,f_1,e_2,f_2,e_3,f_3\}$ of $\mathbb{D}$ with respect to which the actions of the Frobenius and Verschiebung are given by
    \begin{align}\label{eq:F,V,3d}
        F=\begin{pmatrix}
0 & p & 0  & 0  & 0  & 0 \\
1 & 0 & 0  & x &  0 &  z \\
  &   & 0 & p &  0 & 0  \\
  &   & 1 & 0 & 0  & y \\
  &   &   &   & 0 & p \\
  &   &   &   & 1 & 0\\
\end{pmatrix}, \ \ V =\begin{pmatrix}
0 & p &  -\sigma^{-1}(x) & 0  & -\sigma^{-1}(z) & 0 \\
1 & 0 &  0 & 0  & 0  & 0  \\
  &   & 0 & p & -\sigma^{-1}(y) & 0   \\
  &   & 1 & 0 &  0 & 0  \\
  &   &   &   & 0 & p \\
  &   &   &   & 1 & 0\\
\end{pmatrix}. 
    \end{align}  for some $x,y, z \in W$ and such that $\varphi(\langle e_1,f_1\rangle) \subseteq \langle e_1',f_1'\rangle$, $\varphi(\langle e_1,f_1, e_2,f_2\rangle) \subseteq \langle e_1',f_1',e_2',f_2'\rangle$ and $\varphi(\langle e_1,f_1,e_2,f_2,e_3,f_3\rangle) \subseteq \langle e_1',f_1',e_2',f_2',e_3',f_3'\rangle$. 

Furthermore, if there exist admissible filtrations $\Fil \subset \mathbb{D}$ and $\Fil_1, \Fil_2$, $\Fil_3 \subset \mathbb{D}_0$ such that $\varphi(\Fil) \subseteq \Fil_1\oplus \Fil_2 \oplus \Fil_3$, then we may additionally insist that \begin{itemize}
    \item $\varphi(e_1) = p^r e_1'$ and $\varphi(f_1) = p^r f_1'$ for some $r \geq 0$, 
    \item $\varphi(e_2) \in p^s e_2' + \langle e_1', f_1' \rangle$ and $\varphi(f_2) \in p^s f_2' + \langle e_1', f_1' \rangle$, for some $s \geq 0$, 
    \item $\varphi(e_3) \in p^t e_3' + \langle e_1', f_1', e_2' , f_2' \rangle$ and $\varphi(f_3) \in p^t f_3' + \langle e_1', f_1', e_2', f_2' \rangle$ for some $t \geq 0$.
\end{itemize} 
\end{prop}
We also record the following immediate corollary.
\begin{cor}\label{cor:explicit-frob-ver-3d}
Let $A_0$ be a supersingular abelian threefold over $\overline{\mb{F}}_p$. Then there exists a basis of the Dieudonn\'e module of $A_0[p^\infty]$ with respect to which the actions of Frobenius and Verschiebung are given as in \ref{eq:F,V,3d} above for some $x, y, z \in W$.
\end{cor}

For $x, y, z \in W$, let us denote by $(\mathbb{D}(x,y,z),F,V)$ the Dieudonn\'e module with basis $\{e_1,f_1,e_2,f_2,e_3,f_3\}$ with respect to which the Frobenius $F$, and Verschiebung $V$, are given as in \autoref{eq:F,V,3d} of Proposition \ref{prop:explicit-isogeny}. 

\begin{prop}\label{prop:constructible}
\begin{enumerate}
    \item Let $x , y, z \in W$ with $x , y \in W^\times$. Suppose that $x',y',z' \in W$ are such that there exists an isomorphism of Dieudonn\'e modules $\mathbb{D}(x, y, z) \cong \mathbb{D}(x', y', z').$ Then $x' , y' \in W^\times$.
    \item There is a Zariski constructible set $Z \subseteq (\overline{\mb{F}}_p)^3 \times (\overline{\mb{F}}_p)^3$ containing the mod $p$ reduction of $\{(x,y,z,x',y',z') \in W^6 : \mathbb{D}(x,y,z) \cong \mathbb{D}(x',y',z')\}$ and such that the two projections, $Z \xrightarrow{\pi_1} (\overline{\mb{F}}_p)^3$ and $Z \xrightarrow{\pi_2} (\overline{\mb{F}}_p)^3$ have finite fibers of size at most $p^{26}.$
   %
\end{enumerate}
\end{prop}
\begin{proof}
    Let us denote by $\{e_1',f_1',e_2',f_2',e_3',f_3'\}$ the basis of $\mathbb{D}(x',y',z')$ with respect to which the Frobenius $F'$ is as in \autoref{eq:F,V,3d} of Proposition~\ref{prop:explicit-isogeny} (with $x', y', z'$ in place of $x, y, z$). 
    
    The Dieudonn\'e modules $\mathbb{D}(x,y,z)$ and $\mathbb{D}(x',y',z')$ are isomorphic if and only if there exists an $M \in \GL_6{(W)}$ such that $M^{-1}F \sigma(M) = F'.$ 
    Suppose that the first, third and fifth columns of $M$ are respectively: \[\begin{pmatrix}
        A_1\\A_2\\A_3\\A_4\\A_5\\A_6
    \end{pmatrix}, \begin{pmatrix}
        B_1\\B_2\\B_3\\B_4\\B_5\\B_6
    \end{pmatrix}, \text { and } \begin{pmatrix}
        C_1\\C_2\\C_3\\C_4\\C_5\\C_6
    \end{pmatrix}.\]


    It is a straightforward computation that the equation $F \sigma(M) = M F'$ holds if and only if the following set of equations hold.
    \begin{align}
        \sigma^2(A_1) -A_1 &= -\sigma(x)\sigma^2(A_4) \tag{I.i}\\
        p\left[ \sigma^2(A_2)-A_2\right] &= -xA_3-zA_5 \tag{I.ii}\\
        \sigma^2(A_3)-A_3 &= -\sigma(y)\sigma^2(A_6) \tag{I.iii}\\
        p\left[ \sigma^2(A_4)-A_4 \right] &= -y\sigma^2(A_5) \tag{I.iv}\\
        \sigma^2(A_5) &= A_5 \tag{I.v}\\
        \sigma^2(A_6) &= A_6 \tag{I.vi}\\
        \sigma^2(B_1)-B_1 &= x'\sigma(A_2)-\sigma(x)\sigma^2(B_4)-\sigma(z)\sigma^2(B_6) \tag{II.i}
    \end{align}
    (II.ii)\begin{align*}
        p\left[\sigma^2(B_2)-B_2 \right] =\, &x'\sigma^{-1}(A_1)-x\left[\sigma^2(B_3)+\sigma(y)\sigma^2(B_6) \right] \\ &+x'z\sigma(A_6)-z\sigma^2(B_5) 
    \end{align*}
    \begin{align}
        \sigma^2(B_3)-B_3 &= x'\sigma(A_4) -\sigma(y)\sigma^2(B_6) \tag{II.iii}\\
        p\left[ \sigma^2(B_4)-B_4\right] &= x'\left[ \sigma(A_3)+y\sigma(A_6)\right]-y\sigma^2(B_5) \tag{II.iv} \\
        \sigma^2(B_5)-B_5 &= x'\sigma(A_6) \tag{II.v} \\
        p\left[ \sigma^2(B_6) - B_6\right] &= x'\sigma(A_5) \tag{II.vi} \\
        \sigma^2(C_1)-C_1 &= y'\sigma(B_2)-\sigma(x)\sigma^2(C_4)+z'\sigma(A_2)-\sigma(z)\sigma^2(C_6) \tag{III.i} 
    \end{align}
    (III.ii)\begin{align*}
        p\left[ \sigma^2(C_2)-C_2\right] =\,  &y'\left[ \sigma(B_1)+x\sigma(B_4)+z\sigma(B_6)\right] \\ &+z'\left[\sigma(A_1)+x\sigma(A_4)+z\sigma(A_6) \right]\\ &-x\left[\sigma^2(C_3)+\sigma(y)\sigma^2(C_6) \right] -z\sigma^2(C_5)  
    \end{align*}
    \begin{align}
        \sigma^2(C_3)-C_3 &= y'\sigma(B_4)+z'\sigma(A_4)-\sigma(y)\sigma^2(C_6) \tag{III.iii} 
    \end{align}
    (III.iv)\begin{align*}
        p\left[\sigma^2(C_4)-C_4 \right] =\, &y'\left[\sigma(B_3)+y\sigma(B_6)\right] \\ &+z'\left[\sigma(A_3)+y\sigma(A_6) \right]-y\sigma^2(C_5) 
    \end{align*}
    \begin{align}    
        \sigma^2(C_5)-C_5 &= y'\sigma(B_6)+z'\sigma(A_6) \tag{III.v}\\
        p\left[\sigma^2(C_6)-C_6 \right] &= y'\sigma(B_5)+z'\sigma(A_5) \tag{III.vi}
    \end{align}
    
    The equation (I.iv) implies that $p | A_5$, and thus from (I.ii) we get that $p | A_3$, and then from (I.iii) that $p | A_6$. This and equation (II.iv) then imply that $p| B_5.$ 
    We note that $M \in \GL_6(W)$ if and only if its mod $p$ reduction $\overline{M}$ is in $\GL_6(\overline{\mb{F}}_p).$ It is easy to see from a computation of the determinant of $\overline{M}$, thus that $M \in \GL_6(W)$ if and only if $\left[A_1\cdot(\sigma(B_3)+y\sigma(B_6))\cdot C_5\cdot B_3\right] \in W^\times.$
    From (II.ii) and (III.iv) now we see that $x' , y' \in W^\times$. This proves the first part of the proposition.
    
    Furthermore, (II.ii) implies that \begin{equation*}
        x'\sigma^{-1}(A_1) \equiv x\left[\sigma^2(B_3)+\sigma(y)\sigma^2(B_6) \right] \pmod{p} 
    \end{equation*} The above equation, with (I.i) and (II.iii), implies that $x'\left[\sigma(A_1)+x\sigma(A_4) \right] \equiv x\left[ x'\sigma(A_4)+ B_3\right]\pmod{p}.$ Thus, 
    \begin{equation}
        x'\sigma(A_1)  \equiv x B_3 \pmod{p} \tag{$\dagger$}.
    \end{equation} In particular, $\overline{x'} =  ({\overline{x}\overline{B_3}})/({\overline{A_1}})^p.$

    For ease of notation, let us set $A_5' \defeq A_5/p \in W(\mb{F}_{p^2}), A_6' \defeq A_6/p \in W(\mb{F}_{p^2})$ and $A_3' \defeq A_3/p \in W.$ It is easy to see that for fixed $x, y, z$, there are at most $p^{12}$ possibilities for $(\overline{A_1} , \overline{A_2} ,\overline{A_3'} , \overline{A_4} , \overline{A_5'} , \overline{A_6'})$ in $(\overline{\mb{F}}_p)^6$. We now consider two cases.
    
    \textbf{Case 1:} Suppose that $A_4 \equiv 0 \pmod{p}.$
    
    Then (I.iv) implies that $p^2|A_5$ and by (II.vi) we get that $\overline{B_6} \in \mb{F}_{p^2}.$ (II.iii) implies then that there are at most $p^4$ possibilities (for fixed $x,y,z$) in $(\overline{\mb{F}}_p)^2$ for the tuple $(\overline{B_3},\overline{B_6})$.  In Case 1, we have by (I.i) that $\overline{A_1} \in \mb{F}_{p^2}$ and thus recalling that $\overline{x'} =  ({\overline{x}\overline{B_3}})/({\overline{A_1}})^p$, we see that there are at most $p^6$ many possibilities in Case 1, for $(\overline{A_1} , \overline{B_3} , \overline{B_6} , \overline{x'}) \in (\overline{\mb{F}}_p)^4.$ 
    
    \textbf{Case 2:} Suppose that $A_4 \in W^\times.$ 
    
    From $(\dagger)$, we recall that $x' - \frac{x B_3}{\sigma(A_1)} \equiv 0 \pmod{p}.$ Using this along with (II.iii) we get that
    \begin{equation}
        \sigma(y)\sigma^2(B_6) \equiv \left(1+x \sigma\left(\frac{A_4}{A_1} \right)\right)B_3 - \sigma^2(B_3) \pmod{p} \tag{$\ddagger$}
    \end{equation}
    From (II.vi) and (II.iii) we get that \begin{equation*}
        \left[\sigma(y)\sigma(A_5')-\sigma(A_4) \right]\sigma^2(B_6)+\sigma(A_4)B_6 + \sigma(A_5')\sigma^2(B_3)-\sigma(A_5')B_3 = 0.
    \end{equation*}
    Using $(\ddagger)$ now to eliminate $B_6$ from the above equation we get that \begin{align*}
        0 \equiv\, &\sigma^2(B_3)\left(\frac{\sigma(A_4)}{\sigma(y)}\right) \\ &+B_3\left[ \frac{\left(\sigma(y)\sigma(A_5')-\sigma(A_4)\right)\cdot (1+x\sigma(A_4/A_1))}{\sigma(y)} - \frac{\sigma(A_4)}{\sigma^{-1}(y)}\right] \\ &+\sigma^{-2}(B_3)\left[\frac{\sigma(A_4)\sigma^{-2}(1+x\sigma(A_4/A_1))}{\sigma^{-1}(y)} \right] \pmod{p}.
    \end{align*}
We note again that for fixed $(\overline{x},\overline{y},\overline{z}) \in (\overline{\mb{F}}_p)^3$, one sees from the first set of equations (I.i)-(I.vi) that $(\overline{A_1},\overline{A_4},\overline{A_5'}) \in (\overline{\mb{F}}_p)^3$ has at most $p^6$ possibilities and depends algebraically on $(\overline{x},\overline{y},\overline{z})$. For each such choice of $(\overline{A_1},\overline{A_4},\overline{A_5'})$, we have at most $p^4$ choices for $\overline{B_3}$ which again depends algebraically on $(\overline{A_1},\overline{A_4},\overline{A_5'})$ and we note that by $(\ddagger)$, $\overline{B_6}$ is then determined by $\overline{B_3}$ and since $\overline{x'} = \frac{\overline{x}\overline{B_3}}{(\overline{A_1})^p}$, we see that this in turns determines $\overline{x'}$ as well. 

Therefore, in all, combining Case 1 and Case 2, we see that the mod $p$ reduction of \[\{(x,y,z,A_1,A_4,A_5',B_3,B_6,x') \in W^9 : x,y \in W^\times \text{ and (I.i)-(III.vi) satisfied}\}\] is contained in a Zariski constructible subset of $(\overline{\mb{F}}_p)^9$ whose projection to the first three coordinates is quasi-finite with fibers of size at most $p^{11}.$ 

Similarly, from equation (III.iv), we have that \begin{equation*}
    y'\left[\sigma(B_3)+y\sigma(B_6) \right] \equiv y \sigma^2(C_5) \pmod{p}. 
\end{equation*}
We recall that $\sigma(B_3)+y\sigma(B_6)\in W^\times$. If $B_6 \equiv 0 \pmod{p},$ then (III.v) implies that $\overline{C_5} \in \mb{F}_{p^2},$ in which case we have at most $p^{13}$ choices for \[(\overline{A_1},\overline{A_4},\overline{A_5'},\overline{B_3},\overline{B_6},\overline{C_5},\overline{x'},\overline{y'}) \in (\overline{\mb{F}}_p)^{8}.\]

If on the other hand, $B_6 \in W^\times$, then from (III.v) we have \[\frac{\sigma^2(C_5)-C_5}{\sigma(B_6)\sigma^2(C_5)} \equiv \frac{y}{\sigma(B_3)+y\sigma(B_6)} \pmod{p}.\] This implies that for each fixed choice of $(\overline{B_3},\overline{B_6},\overline{y})$ there are at most $p^2$ possibilities for $\overline{C_5},$ which in turn determines $\overline{y'}.$ Therefore, putting everything together we see that there is a Zariski constructible subset of $\overline{\mb{F}}_p$ containing the mod $p$ reduction of $\{(x,y,z,A_1,A_4,A_5',B_3,B_6,x',C_5,y') : x,y \in W^\times \text{ and (I.i)-(III.vi) hold}\}$ and such that the projection to the first three coordinates has fibers of size at most $p^{14}.$

(III.ii) implies that \begin{align*}
    z'\left[\sigma^{-1}(A_1) \right] \equiv\, &y'\left[\sigma(B_1)+x\sigma(B_4)+z\sigma(B_6) \right] \\ &+ x\left[\sigma^2(C_3)+\sigma(y)\sigma^2(C_6) \right] + z \sigma^2(C_5) \pmod{p}.
\end{align*}
From (III.iii), (III.v), and (III.vi), we see that the reduction of the set \begin{align*} \{(x,y,z,x',y',z',A_1,A_4,A_5',A_6',B_1,B_3,&B_4,B_5',B_6,C_3,C_5,C_6)\in W^{18} : \\ &x,y \in W^\times \text{ and (I.i)-(III.vi) hold}\} \end{align*} 
is contained in a Zariski constructible subset of $(\overline{\mb{F}}_p)^{18}$ whose projection to the first three coordinates has fibers of size at most $p^{26}.$ This completes the proof of the Proposition.
\end{proof}

\section{Computations}
\subsection{Lifting abelian surfaces}
Let $A_0$ be a supersingular abelian surface over $\overline{\mb{F}}_p$, and we write its Dieudonn\'e module in the form given by Proposition~\ref{prop:explicit-isogeny-2d} for some particular $x \in W(\overline{\mb{F}}_p)$. If $x$ is divisible by $p$, then $A_0$ is itself the square of an elliptic curve, and the result is trivial in this case. We assume from now on that $x$ is a unit, so that we can find $\beta \in W^{\times}$ satisfying 
\[
\sigma^2(\beta)-\beta = x.
\]
Let $\zeta \in W^{\times}$ be a primitive $(p^2-1)$-th root of unity, and write $u\defeq \sigma(\zeta)-\zeta$.   

The following is straightforward by an application of Hensel's lemma.
\begin{prop}\label{prop:henselsurface}
There exists $b\in W^{\times}$ such that 
\[
p\beta b^2+ub -\sigma^{-1}(\beta)=0.
\]
\end{prop}

Let $\mb{D}'$ denote the Dieudonn\'e module of a product of supersingular elliptic curves. We will work with a basis $e_1', f_1', e_2', f_2'$ such that the  actions of Frobenius and Verschiebung are given by 
\[
F'=V'= \begin{pmatrix}
0 & p & & \\
1 & 0 & & \\
 &  & 0 & p \\
 &  & 1 & 0\\
\end{pmatrix}.
\]
Moroever, consider the following filtration, where $b$ is as in Proposition~\ref{prop:henselsurface}
\[
\Fil'\defeq \langle f_1'+pbe_1', f_2'+pbe_2' \rangle
\]
on $\mb{D}'$, which corresponds to a square of an elliptic curve over $W$, which we denote by $\tilde{E}\times \tilde{E}$. 
The following now finishes the proof of the first part of Theorem~\ref{thm:main}.
\begin{claim}
The filtration on $\mb{D}$ given by 
\[
\Fil \defeq \langle f_1+pbe_1, f_2+pbe_2-\sigma^{-1}(x)e_1 \rangle 
\] 
gives rise to a lift of $A$ which is isogenous to $\tilde{E} \times \tilde{E}$.
\end{claim} 

\begin{proof}[Proof of Claim]
By Proposition~\ref{reduction-to-d-modules}, it suffices to show that the filtered Dieudonn\'e modules $(\mb{D}, \Fil)$ and $(\mb{D}', \Fil')$ are isogenous.  

Consider the map $\varphi: \mb{D}\rightarrow \mb{D}'$ given by 
\begin{align*}
    e_1 &\mapsto pe_1' \\
    f_1 &\mapsto pf_1'\\
    e_2 &\mapsto \zeta e_1'+\beta f_1' +e_2' \\
    f_2 &\mapsto p\sigma(\beta)e_1' +\sigma(\zeta)f_1' +f_2'.
\end{align*}
It remains  to check that $\varphi$ commutes with Frobenius and Verschiebung, and sends $\Fil$ to $\Fil'$; this is a straightforward computation.
\end{proof}

\subsection{Lifting abelian threefolds}
In this subsection, we prove that every supersingular abelian threefold over $\overline{\mb{F}}_p$, admits a lift to an abelian scheme over $W$ such that the lift is isogenous to a product of three elliptic curves.

Let $A_0$ be a supersingular abelian threefold over $\overline{\mb{F}}_p.$ Let $(\mathbb{D},F,V)$ denote the Dieudonn\'e module of the $p$-divisible group $A_0[p^\infty]$ associated to $A_0$. By Corollary \ref{cor:explicit-frob-ver-3d}, we may choose a $W$-basis $\{e_1,f_1,e_2,f_2,e_3,f_3\}$ of $\mathbb{D}$ with respect to which the Frobenius and Verschiebung operators are as described in \autoref{eq:F,V,3d} of Proposition  \ref{prop:explicit-isogeny} for some $x, y, z \in W$.

Pick an element $\beta' \in W^\times$ such that \[\sigma^2(\beta')-\beta' = x.\]

Pick an element $Q \in W^\times$ such that \[\sigma^2(Q)-Q = py\sigma(\beta').\] Note that if $Q \in W^\times$ satisfies the equation $\sigma^2(Q)-Q = py\sigma(\beta')$ then so does $Q+u$ for any $u \in W(\mb{F}_{p^2}).$ Thus, we may also insist that $1-Q \in W^\times$ and $Q - \sigma(Q) \in W^\times.$

Pick an element $R$ in $W^\times$ such that \[\sigma^2(R)-R = y+pz\] and such that $R - \beta'Q \in W^\times.$

Pick an element $T \in W^\times$ such that \[\sigma^2(T)-T = y.\] Again, if $T \in W^\times$ satisfies the equation $\sigma^2(T)-T = y$ then so does $T+u$ for every $u\in W(\mb{F}_{p^2})$. Thus, we may insist additionally that $(T - \beta'Q) \in W^\times$ and $(-Q(T-1)+R-\beta'Q) \in W^\times$ and also that $(\sigma^{-1}(Q)(T-1)-(R-\beta'Q)) \in W^\times.$

To simplify the notations in the computations that follow, we introduce certain auxiliary elements $\lambda_1,\mu_1,\nu_1,\Delta_1,\nu_2, \Delta_2,\Phi, \psi, \theta, S \in W$ defined as follows:

\begin{equation*}
    \begin{aligned}[c]
    \lambda_1 &:= pR\sigma^{-1}(\beta')-Q, \\
    \mu_1 &:= \sigma^{-1}(Q)-p\beta'\sigma^{-1}(R), \\
    \nu_1 &:= p(R-\beta' Q), \\
    \Delta_2 &:= \sigma^{-1}(T)-\sigma^{-1}(S),\\
    \Phi &:= \Delta_1\nu_2 - \Delta_2 \nu_1, \\
    \end{aligned}
    \begin{aligned}[c]
    \Delta_1 &:= \sigma^{-1}(R)-\sigma^{-1}(Q)\sigma^{-1}(\beta')\\
    S &:= 1+pT, \\
    \nu_2 &:= p(T-S),\\
        \psi &:= \mu_1 \nu_2 - \nu_1, \\
        \theta &:= \lambda_1\nu_2 + \nu_1.
        \end{aligned}
\end{equation*}

Note that $\psi = p[(\sigma^{-1}(Q)-p\beta'\sigma^{-1}(R))(T-S)-(R-\beta'Q)]$ and $\theta = p[(pR\sigma^{-1}(\beta')-Q)(T-S)+(R-\beta'Q)]$. We have that $1 = \mathrm{ord}_p(\psi) = \mathrm{ord}_p(\nu_1) = \mathrm{ord}_p(-\lambda_1\psi+\mu_1\theta) = \mathrm{ord}_p(\theta)$ and we note that $\Phi \in p W.$

By an application of Hensel's Lemma, we may find a $c \in W$, such that \[c^2[-\nu_1\psi]+c[-\lambda_1\psi+\mu_1\theta+\nu_1\Phi]+[\lambda_1\Phi-\Delta_1\theta] = 0.\]

Set $i := \frac{\Phi-\psi c}{\theta}.$ Note that $i \in W$ since $\mathrm{ord}_p(\theta) = 1.$ 

Set $a := \frac{\sigma^{-1}(R)+iQ}{\sigma^{-1}(Q)+piR}.$ We note that $a \in W$ since $\sigma^{-1}(Q)+piR \in W^\times.$
One may also check that $a = \frac{\sigma^{-1}(\beta')+c}{1+pc\beta'}.$ 

We let $b' := \frac{\sigma^{-1}(T)+iS}{\sigma^{-1}(S)+piT}.$ Note that $b' \in W$ since $\sigma^{-1}(S)+piT \in W^\times.$ We may also check that $b' =\frac{1+c}{1+pc}.$

Let $\Fil \subseteq \mathbb{D}$ be the following admissible filtration \[\Fil \defeq \langle f_1+pae_1, f_2 - \sigma^{-1}(x)e_1+pce_2, f_3-\sigma^{-1}(y)e_2+pie_3\rangle.\]

Let $E_0$ be a supersingular elliptic curve over $\overline{\mb{F}}_p$ and let $(\mathbb{D}_0, F_0, V_0)$ denote the Dieudonn\'e module of $E_0[p^\infty]$. There is a $W$-basis $\{e,f\}$ of $\mathbb{D}_0$ such that $F_0(e) = f = V_0(e)$ and $F_0(f) = pe = V_0(f).$ 
As in \S~\ref{sec:lifitng-threefolds}, we set $(\mathbb{D}',F',V') := (\mathbb{D}_0,F_0,V_0)^{\oplus 3}$ and let $e_1' \defeq (e,0,0), f_1' \defeq (f,0,0), e_2'\defeq (0,e,0),f_2'\defeq (0,f,0), e_3'\defeq (0,0,e)$ and $f_3'\defeq(0,0,f)$.

Let $\Fil' \subseteq \mathbb{D}'$ be the filtration \[\Fil' \defeq \langle f_1'+pae_1', f_2'+pb'e_2', f_3'+pie_3'\rangle. \] 

Consider the $W$-linear map $\varphi : \mathbb{D} \rightarrow \mathbb{D}'$ defined as follows:
\begin{align*}
    \varphi : \mathbb{D} &\rightarrow \mathbb{D}'\\
    e_1 &\mapsto p^2 e_1',\\
    f_1 &\mapsto p^2f_1',\\
    e_2 &\mapsto pe_1'+p\beta'f_1' + pe_2' + pf_2',\\
    f_2 &\mapsto pf_1' + p^2\sigma(\beta')e_1'+pf_2'+p^2e_2',\\
    e_3 &\mapsto Qe_1'+Rf_1'+Se_2'+Tf_2'+e_3',\\
    f_3 &\mapsto \sigma(Q)f_1'+p\sigma(R)e_1'+\sigma(S)f_2'+p\sigma(T)e_2'+f_3'.
\end{align*}

It is a straightforward computation to verify that $\varphi$ defines an isogeny of filtered Dieudonn\'e modules $\varphi : (\mathbb{D},\Fil,F,V) \rightarrow (\mathbb{D}',\Fil',F',V')$. We conclude by  Proposition~\ref{reduction-to-d-modules}.

\subsection{Impossibility of lifting general supersingular threefolds to varieties isogenous to cubes}

Let $A_0$ denote a supersinguar abelian threefold over $\overline{\mb{F}}_p$ and $(\mathbb{D},F,V)$ the Dieudonn\'e module of $A_0[p^\infty]$. 

\begin{prop}\label{prop:lin-dependence}
Suppose there is an abelian scheme $\mathcal{A}$ over $\Spec(W)$ that lifts $A_0$, and an elliptic curve $\mathcal{E}$ over $\Spec(W)$ with an isogeny $\varphi :  \mathcal{A} \rightarrow \mathcal{E}^3$ over $W$. Then there is a $W$-basis $\{e_1,f_1,e_2,f_2,e_3,f_3\}$ of $\mathbb{D}$ with respect to which the action of the Frobenius is given by 

    \[F = \begin{pmatrix}
0 & p &   &   &   &  \\
1 & 0 &   & x &   & z  \\
  &   & 0 & p &   &   \\
  &   & 1 & 0 &   & y \\
  &   &   &   & 0 & p \\
  &   &   &   & 1 & 0\\
\end{pmatrix}\] for some $x , y, z \in W$ such that the reductions $\overline{x},\overline{y}$ in $\overline{\mb{F}}_p$ are linearly dependent over $\mb{F}_{p^2}$.
\end{prop}
\begin{proof}
    
    Let $\Fil \subset \mathbb{D}$ denote the filtration on $\mathbb{D}$ corresponding to the lift $\mathcal{A}[p^\infty]$ of $A_0[p^\infty]$ and let $(\mathbb{D}',\Fil',F',V')$ be the filtered Dieudonn\'e module corresponding to $\mathcal{E}^3[p^\infty].$ We choose a $W$-basis $\{e_1',f_1'.e_2',f_2',e_3',f_3'\}$ of $\mathbb{D}'$ such that $\Fil' = \langle f_1'+pae_1', f_2'+pae_2', f_3'+pae_3' \rangle$ for some $a \in W$ and such that $F'e_i' = f_i'$ and $F'f_i' = pe_i'$ for $i\in \{1,2,3\}.$
    
    We denote by $\varphi : (\mathbb{D},\Fil,F,V) \rightarrow (\mathbb{D}',\Fil',F',V')$ the isogeny of filtered Dieudonn\'e modules corresponding to the isogeny $\mathcal{A}\rightarrow \mathcal{E}^3.$
    By Proposition~\ref{prop:explicit-isogeny}, we may find a basis $\{e_1,f_1,e_2,f_2,e_3,f_3\}$ of $\mathbb{D}$ with respect to which \[F = \begin{pmatrix}
0 & p &   &   &   &  \\
1 & 0 &   & x &   & z  \\
  &   & 0 & p &   &   \\
  &   & 1 & 0 &   & y \\
  &   &   &   & 0 & p \\
  &   &   &   & 1 & 0\\
\end{pmatrix}\] for some $x,y , z \in W$ and furthermore such that the filtration $\Fil \subset \mathbb{D}$ and the isogeny $\varphi$ have the following shape: 
\begin{itemize}
    \item $\Fil = \langle f_1+pa'e_1,f_2-\sigma^{-1}(x)e_1+pbe_1+pce_2, f_3-\sigma^{-1}(y)e_2-\sigma^{-1}(z)e_1+pde_1+pge_2+phe_3 \rangle$, for some $a',b,c,d,g,h \in W$,  
    \item    \begin{align*} \varphi : \mathbb{D} &\rightarrow \mathbb{D}' \\
                    e_1 &\mapsto p^r e_1'\\
                    f_1 &\mapsto p^r f_1'\\
                    e_2 &\mapsto \alpha e_1' + \beta f_1' + p^s e_2'\\
                    f_2 &\mapsto \sigma(\alpha)f_1' + p\sigma(\beta)e_1' + p^s f_2' \\
                    e_3 &\mapsto Qe_1' + Rf_1' +Se_2' + Tf_2' + p^t e_3'\\
                    f_3 &\mapsto \sigma(Q)f_1' + p\sigma(R) e_1' + \sigma(S)f_2' + p\sigma(T)e_2' + p^t f_3'
    \end{align*} for integers $r, s, t \geq 0$ and where $\alpha, \beta, Q, R, S, T \in W.$
\end{itemize}

We note that $\varphi(f_1+pa'e_1) \in \Fil'$ if and only if $a' = a$ and thus henceforth we assume that $a' = a$. It is a straightforward calculation that $\varphi$ defines an isogeny of filtered Dieudonn\'e modules if and only if the following equations hold: 

\begin{align}
    \sigma^2(\alpha) &= \alpha \tag{1.i}\\
    \sigma^2(\beta)-\beta &= p^{r-1}x \tag{1.ii}\\
    a[\sigma(\alpha)+pc\beta] &= [\sigma^{-1}(\beta)+p^rb + c\alpha] \tag{2.i}\\
    a &= c \tag{2.ii}\\
    \sigma^2(Q) - Q &= y\sigma(\beta) \tag{3.i} \\
    p[\sigma^2(R)-R] &= y\sigma(\alpha) +p^rz \tag{3.ii}\\
    \sigma^2(S) &= S \tag{3.iii} \\
    \sigma^{2}(T)-T &= p^{s-1}y \tag{3.iv}\\
    a &= h \tag{4.i}\\
    a[\sigma(S)+paT] &= \sigma^{-1}(T)+p^sg+aS \tag{4.ii}\\
    a[\sigma^{-1}(Q)+pg\beta+paR] &= [\sigma^{-1}(R)+g\alpha + p^rd+aQ]. \tag{4.iii}
\end{align}

We may assume that $x , y \in W^\times$ since otherwise the conclusion of the Proposition evidently holds. 
From (1.ii), we see that $r \geq 1$.
Then (3.ii) implies that $p$ divides $\alpha$ and thus from (2.ii) we see that $p$ divides $\beta$ as well. Then (1.ii) implies that $r \geq 2$. Let $\alpha = p\alpha_0$ and $\beta = p\beta_0.$ The above equations can be rewritten to give the following equations:
\begin{align}
    \sigma^2(\alpha_0) &= \alpha_0 \tag{1.i'}\\
    \sigma^2(\beta_0) -\beta_0 &= p^{r-2}x \tag{1.ii'}\\
    a[\sigma(\alpha_0)+pa\beta_0] &= \sigma^{-1}(\beta_0) +p^{r-1}b + a\alpha_0 \tag{2.i'}\\
    \sigma^2(Q) - Q &= py\sigma(\beta_0) \tag{3.i'} \\
    \sigma^2(R)-R &= y\sigma(\alpha_0) +p^{r-1}z \tag{3.ii'}\\
    \sigma^2(S) &= S \tag{3.iii'} \\
    \sigma^{2}(T)-T &= p^{s-1}y \tag{3.iv'}\\
    a[\sigma(S)+paT] &= \sigma^{-1}(T)+p^sg+aS \tag{4.ii'}\\
    a[\sigma^{-1}(Q)+p^2g\beta_0+paR] &= [\sigma^{-1}(R)+pg\alpha_0 + p^rd+aQ]. \tag{4.iii'}
\end{align}
  
  The equation (2.i') implies that $v(\beta_0) = v(\sigma(\alpha_0)-\alpha_0) + v(a)$ and similarly (4.ii') and (4.iii') imply $v(T) = v(\sigma(S)-S)+v(a)$ and $v(R) = v(\sigma^{-1}(Q)-Q)+v(a).$ Suppose that $m \defeq v(a)$, $k \defeq v(T)$ and $n \defeq v(\beta_0)$. Then $v(\sigma(\alpha_0)-\alpha_0) = n-m$ and $v(\sigma(S)-S) = k-m.$  We note that $r-2 \geq n \geq 2m$.
  We write $a = p^ma_1$, $\beta_0 = p^n \beta_1$ and $T = p^k T_1$, with $a_1, \beta_1, T_1 \in W^\times.$
  
  We rewrite (2.i') and (4.ii') to get \begin{align}
      a_1\left[\frac{\sigma(\alpha_0)-\alpha_0}{p^{n-m}} + p^{2m+1}a_1\beta_1\right] &= \sigma^{-1}(\beta_1)+p^{r-1-n}b, \tag{$\star$}\\
      a_1\left[\frac{\sigma(S)-S}{p^{k-m}}+p^{2m+1}a_1T_1\right] &= \sigma^{-1}(T_1) + p^{s-k}g. \tag{$\star\star$}
  \end{align}
  
  Let \begin{align*}
      u_1 &\defeq \frac{\sigma(\alpha_0)-\alpha_0}{p^{n-m}} \in W(\mb{F}_{p^2})^\times \\
      u_1' &\defeq \frac{\sigma(S)-S}{p^{k-m}} \in W(\mb{F}_{p^2})^\times.
  \end{align*}
  Rewriting the equations $(\star)$ and $(\star\star)$ we get that \begin{align}
      a_1^2(p^{2m+1}\beta_1) + a_1u_1 - \sigma^{-1}(\beta_1) &= p^{r-1-n}b \tag{$\star'$}\\
      a_1^2(p^{2m+1}T_1)+a_1u_1'-\sigma^{-1}(T_1) &= p^{s-k}g \tag{$\star\star'$}
  \end{align} 
  
  \begin{lem}\hfill
  \begin{enumerate}
      \item $\left( a_1 - \frac{\sigma^{-1}(\beta_1)}{u_1}\right) \pmod{p^{r-1-n}} \in W(\mb{F}_{p^2})/\langle p^{r-1-n}\rangle$ 
      \item $\left( a_1 - \frac{\sigma^{-1}(T_1)}{u_1'}\right) \pmod{p^{s-k}} \in W(\mb{F}_{p^2})/\langle p^{s-k}\rangle$
  \end{enumerate}
  \end{lem}
  \begin{proof}
      Let $\Delta_1, \Delta_2$ be the two roots of the equation $X^2(p^{2m+1}\beta_1)+X(u_1)-\sigma^{-1}(\beta_1) = 0$ with $v(\Delta_1) = 0$ and $v(\Delta_2) = -(2m+1).$ 
      From $(\star')$, we see that $p^{r-1-n}$ divides $p^{2m+1}\beta_1(a_1-\Delta_1)(a_1-\Delta_2)$. Since $\beta_1 p^{2m+1}(a_1-\Delta_2) \in W^\times$ this implies that $p^{r-1-n}$ divides $(a_1-\Delta_1).$
      
      So it remains to prove that \[\left(\Delta_1 - \frac{\sigma^{-1}(\beta_1)}{u_1}\right) \pmod{p^{r-1-n}}\in W(\mb{F}_{p^2})/\langle p^{r-1-n} \rangle. \]
      
      Pick a $\chi \in W^\times$ such that $\sigma^2(\chi) -\chi = x.$ Then we may write $\beta_1 = w_1 + p^{r-2-n}\chi$ with $w_1 \in W(\mb{F}_{p^2}).$ We see that \begin{align*}
          &\Delta_1 = \frac{\sigma^{-1}(\beta_1)}{u_1} + \sum_{j \geq 2} \binom{1/2}{j} 2^{2j-1}p^{(j-1)(2m+1)}\beta_1^{j-1} \sigma^{-1}(\beta_1)^j u_1^{-(2j-1)} \\
          &= \frac{\sigma^{-1}(\beta_1)}{u_1} + \sum_{j \geq 2} \binom{1/2}{j} 2^{2j-1}p^{(j-1)(2m+1)}(w_1+p^{r-2-n}\chi)^{j-1} \sigma^{-1}(w_1+p^{r-2-n}\chi)^j u_1^{-(2j-1)}\\
          &\equiv \frac{\sigma^{-1}(\beta_1)}{u_1} + \sum_{j \geq 2} \binom{1/2}{j} 2^{2j-1}p^{(j-1)(2m+1)}w_1^{j-1} \sigma^{-1}(w_1)^j u_1^{-(2j-1)} \pmod{p^{r-1-n}}
      \end{align*} from which (1) follows since the final sum above is indeed in $W(\mb{F}_{p^2}).$
      
      The proof of (2) is similar to that of (1). 
  \end{proof}
  
  By switching the roles of $\beta_1$ with $T_1$ and $u_1$ with $u_1'$ if need be, in the argument that follows, we may assume that $s-k \leq r-1-n.$ Then from the above Lemma we see that 
  \[\Lambda \defeq \left( \frac{\sigma^{-1}(\beta_1)}{u_1} - \frac{\sigma^{-1}(T_1)}{u_1'}\right) \pmod{p^{s-k}} \in W(\mb{F}_{p^2})/\langle p^{s-k}\rangle.\] In other words, $p^{s-k}$ divides $\sigma^{2}(\Lambda) - \Lambda$
  
  \begin{align*}
      p^{s-k} | \left(\sigma^{2}(\Lambda) - \Lambda \right) &= \left( \frac{\sigma(\beta_1)-\sigma^{-1}(\beta_1)}{u_1}\right) - \left( \frac{\sigma(T_1)-\sigma^{-1}(T_1)}{u_1'}\right)\\
      &= \frac{p^{r-2-n}\sigma^{-1}(x)}{u_1} - \frac{p^{s-k-1}\sigma^{-1}(y)}{u_1'}.
  \end{align*}
  
  This implies that $p$ divides $\left(\sigma(u_1')p^{(r-2-n)-(s-k-1)}x - \sigma(u_1)y\right)$ and therefore that the reductions $\overline{x}, \overline{y} \in \overline{\mb{F}}_p$ are linearly dependent over $\mb{F}_{p^2}.$
\end{proof}

\begin{thm}\label{thm:part3}
There exist infinitely many supersingular abelian threefolds $A_0$ over $\overline{\mb{F}}_p$ such that there does not exist an abelian scheme $\mathcal{A}$ over $\Spec(W)$ lifting $A_0$ and such that $\mathcal{A}$ is isogenous over $W$ to the cube of an elliptic curve. 
\end{thm}
\begin{proof}
Let $Z\subset (\overline{\mb{F}}_p)^3\times (\overline{\mb{F}}_p)^3$ be the constructible set given by Proposition~\ref{prop:constructible}, and  $\pi_1, \pi_2$ the projections  as in the statement of Proposition \ref{prop:constructible}.

Let us denote by $C \subset \overline{\mb{F}}_p \times \overline{\mb{F}}_p$ the plane curve  $C := \{ (a, b) \in \overline{\mb{F}}_p \times \overline{\mb{F}}_p : a^{p^2-1} = b^{p^2-1}\}.$
Let $A_0$ be a supersingular abelian threefold such that the Dieudonn\'e module of $A_0[p^\infty]$ is isomorphic to $\mathbb{D}(x',y',z')$ for some $x', y', z'$ with $x', y' \in W^\times.$

By Proposition \ref{prop:lin-dependence}, if $A_0$ admits a lift to $W$ that is isogenous to a cube of an elliptic curve, then there exist $x,y ,z \in W,$ such that $\mathbb{D}(x',y',z') \cong \mathbb{D}(x,y,z)$ and such that $\overline{x}, \overline{y}$ are linearly dependent over $\mb{F}_{p^2},$ i.e. $(\overline{x},\overline{y}) \in C.$

This in particular implies that $(x',y',z') \in \pi_2\left(\pi_1^{-1}\left(C \times \overline{\mb{F}}_p\right)\right).$ Since $C$ is one-dimensional and since the maps $\pi_1, \pi_2 : Z\rightarrow \left(\overline{\mb{F}}_p \right)^3$ have finite fibers, the subset $\pi_2\left(\pi_1^{-1}\left(C \times \overline{\mb{F}}_p\right)\right)$  of $\left(\overline{\mb{F}}_p \right)^3$ has dimension at most 2. 
This completes the proof of the theorem once we note that for every $(x',y',z') \in W^3$, there is a supersingular abelian threefold over $\overline{\mb{F}}_p$ such that the Dieudonn\'e module associated to its $p$-divisible group is isomorphic to $\mathbb{D}(x',y',z')$.
\end{proof}

\printbibliography[]

\end{document}